\documentclass{amsart}

\pdfoutput=1

\usepackage{newlattice}
\usepackage{graphics}

\theoremstyle{plain}
\newtheorem{theorem}{Theorem}

\newtheorem{lemma}{Lemma}
\newtheorem{definition}{Definition}

\theoremstyle{remark}

\begin{document}

\title{Notes on planar semimodular lattices. I. Construction} 
\author{G. Gr\"{a}tzer} 
\address{Department of Mathematics\\
   University of Manitoba\\
   Winnipeg, MB R3T 2N2\\
   Canada}
\email[G. Gr\"atzer]{gratzer@ms.umanitoba.ca}
\urladdr[G. Gr\"atzer]{http://server.math.umanitoba.ca/homepages/gratzer/}
\thanks{The research of the first author was supported by the NSERC of Canada.}

\author{E. Knapp}
\address{University of Manitoba\\
   Winnipeg, MB R3T 2N2\\
   Canada}
\email[E. Knapp]{edward.m.knapp@gmail.com}
\date{Nov. 27, 2006; revised Feb. 10, 2007}
\keywords{Semimodular lattice, planar, distributive, modular.}
\subjclass[2000]{Primary: 06C10; Secondary: 06D05}
\begin{abstract}
We construct all planar semimodular lattices in three simple steps from the direct product of two chains.
\end{abstract}

\maketitle

\section{Introduction}\label{S:1}
It is part of the folklore of lattice theory that a planar distributive lattice $D$ is a cover-preserving sublattice of a direct product of two finite chains; in fact, it is the direct product with the two ``corners'' removed (a corner may have any number of elements), as illustrated by the lattice $S$ in Fig\emph{}ure~\ref{Fi:distrandmod}. It is also known that we obtain a planar modular lattice $S^+$ from a planar distributive lattice $S$ by adding ``eyes'' to covering squares---making covering squares into covering~$M_n$-s, see the lattice $S^+$ in Figure~\ref{Fi:distrandmod} for an example; the elements of~$S$ in~$S^+$ are black-filled.

\begin{figure}[htb]
\centerline{\includegraphics{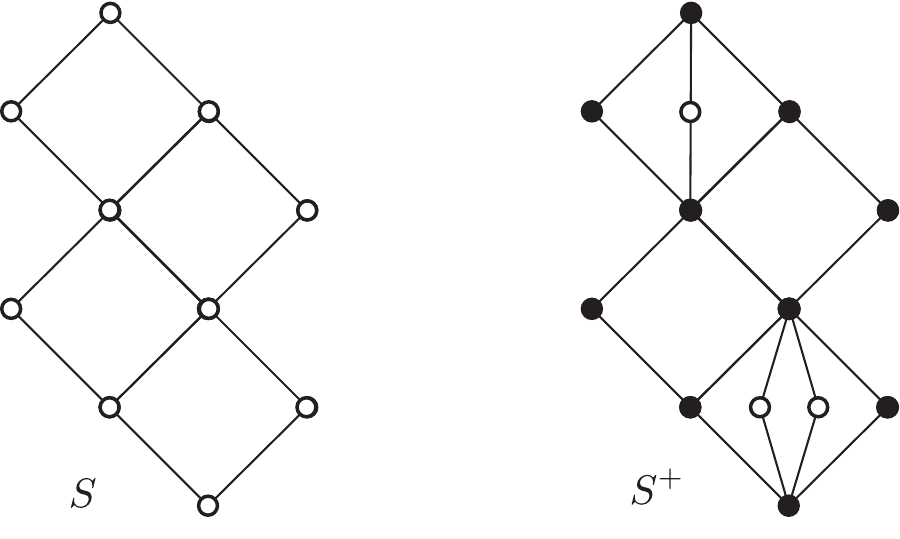}}
\caption{A planar distributive and a planar modular lattice.}\label{Fi:distrandmod}
\end{figure}

A typical example of a planar semimodular lattice $S_7$ is shown in Figure~\ref{Fi:semimod}. There is no obvious way to connect it with a planar distributive lattice. However, we can play with $S_7$ the same game we played before with $S$---adding ``eyes''---and obtain the second lattice $S_7^+$ of Figure~\ref{Fi:semimod}; again, a planar semimodular lattice.

\begin{figure}[htb]
\centerline{\includegraphics{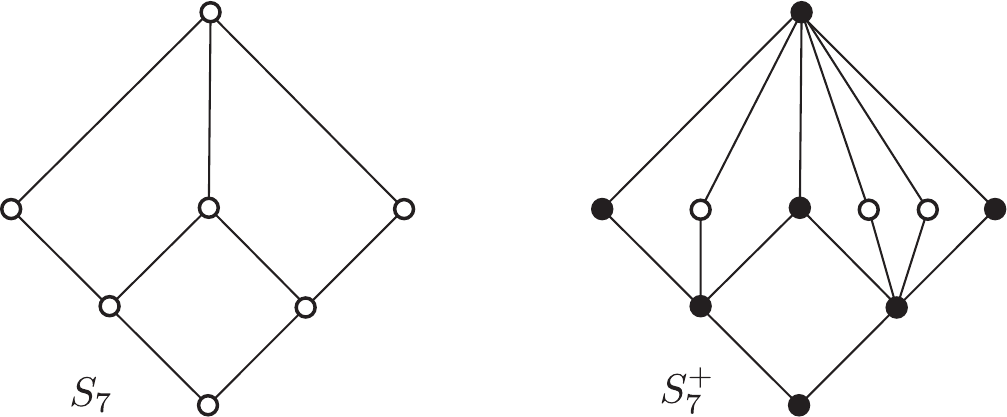}}
\caption{Two planar semimodular lattices.}\label{Fi:semimod}
\end{figure}

In this paper, it is our goal to describe how to construct all planar semimodular lattices.

\subsection*{Acknowledgment} We would like to extend our thanks to David Kelly for his comments on the presentation of this result.

\subsection*{Notation} We follow the notation and terminology of the book \cite{CLFL}; see also\\
\ttt{http://www.maths.umanitoba.ca/homepages/gratzer.html/}\\
and click on \ttt{Notation}.

For a lattice $L$ and elements $a$, $b \in L$, we use the notation $a \preceq b$ for $a \prec b$ or $a  =  b$.

\section{Corners}\label{S:corners}
Let $C$ \tup{(}with zero $0_C$ and unit $1_C$\tup{)} and $D$ \tup{(}with zero $0_D$ and unit $1_D$\tup{)} be finite chains. A \emph{left corner} of $A  =  C \times D$ is defined as follows.  A \emph{left $1$-corner} of $L$ is $\set{\vv<1_C, 0_D>}$ where $L$ is the left boundary chain of $A$. Removing this element, we get the lattice $A_1$. Obviously, $A_1$ is a planar distributive lattice. Let $L_1$ be the left boundary chain of $A_1$.

Having defined a planar distributive lattice $A_{n-1}$ with left boundary chain $L_{n-1}$, pick a doubly irreducible element $a$ of $L_{n-1}$. Define $A_{n}  =  A_{n-1} - \set{a}$ with left boundary chain $L_{n}$. The corresponding left corner is $A - A_n$.

A left corner is an $A - A_n$, for some $n$.

We define right corners similarly.

Now we can restate the folklore result on planar distributive lattices.

\begin{lemma}\label{L:planardistr}
A finite, planar, distributive lattice can be obtained from the direct product of two finite chains by removing a left and a right corner.
\end{lemma}

For more detail on this process, see D. Kelly and I. Rival \cite{KR75}.

\section{Slimming}\label{S:slim}
Let $L$ be a planar lattice (by definition, a planar lattice is finite). Let 
\[
   \set{o, a, b, c, i}  =  M_3
\]
be a cover-preserving sublattice of $S$, with zero $o$, unit $i$, atoms $a$ to the left of $b$, to the left of $c$. Then $L_1  =  L - \set{b}$ is a sublattice of $L$, a \emph{$1$-step slimming} of~$L$. Obviously, $L_1$ is also a planar lattice. In general, an \emph{$n$-step slimming} of $L$ is a $1$-step slimming of an $n - 1$-step slimming of $L$; a \emph{slimming} is an $n$-step slimming, for some $n$. In Figure~\ref{Fi:semimod}, the lattice $S_7$ is a $3$-step slimming of the lattice~$S_7^+$.

We call a planar lattice \emph{slim}, if it has no $1$-step slimming; equivalently, every covering square is an interval. In Figure~\ref{Fi:semimod}, the lattice $A$ is slim, the lattice~$B$ is~not.

The following observation is trivial:

\begin{lemma}\label{L:slimfat}
Let $L$ be a planar lattice and let $\ol{L}$ be a slimming of $L$. Then $L$ is semimodular if{}f $\ol{L}$ is semimodular.
\end{lemma} 

With this terminology, we can recast the folklore result as follows:

\begin{lemma}\label{L:planarmod}
Let $L$ be a planar modular lattice. If $L$ is slim, then it is distributive.
\end{lemma} 

\section{$4$-cell lattices}\label{S:3}

It seems beneficial to look at planar semimodular lattices \emph{via} their cell structure. Cells were introduced in O. Ore \cite{oO43}; see also S. MacLane \cite{sM43}.

\begin{definition}\label{D:cell}
A \emph{cell} $A$ in a planar lattice consists of two maximal chains $C$ \tup{(}with zero $0_C$ and unit $1_C$\tup{)} and $D$ \tup{(}with zero $0_D$ and unit $1_D$\tup{)} such that the following conditions hold:
\begin{enumeratei}
\item $0_C  =  0_D$ and $1_C  =  1_D$;
\item every $x \in C-\set{0_C, 1_C}$ is to the left of every $y \in D - \set{0_D, 1_D}$;
\item there are no elements inside the region bounded by $C$ and $D$.
\end{enumeratei}
We call $C$ the \emph{left chain} and $D$ the \emph{right chain} of the cell $A$.
\end{definition}

 A \emph{$4$-cell} is a cell with $|C|  =  |D|  =  3$. A \emph{$4$-cell lattice} is a lattice in which all cells are $4$-cells.

\begin{lemma}\label{L:4cell1}
A planar semimodular lattice is a $4$-cell lattice.
\end{lemma}
\begin{proof}
Obvious.
\end{proof}

For a cell $A$, let $0_A$ and $1_A$ denote the zero and unit of $A$, respectively.

\begin{lemma}\label{L:4cell2}
Let $L$ be a $4$-cell lattice. Then $L$ is semimodular if{}f for the cells $A$ and $B$, if $0_A  =  0_B$, then $1_A  =  1_B$.
\end{lemma}

\begin{proof}
Let $L$ be a finite planar $4$-cell lattice with cells $A$ and $B$. Let $a_1$ and $a_2$ be the atoms of $A$, from left to right, and let $b_1$ and $b_2$ be the atoms of $B$, from left to right. Let $A$ be to the left of $B$ (that is, $a_1$ is to the left of $b_2$).

If $L$ is semimodular, then $a_1 \jj b_2$ is of height $2$; because of planarity, we conclude that $a_1 \jj b_2  =  1_A  =  1_B$.

Conversely, assume that for any two cells, $A$ and $B$, if $0_A  =  0_B$, then $1_A  =  1_B$. To verify that $L$ is semimodular, let $a$, $b$, $o \in L$, $o \prec a$, $o \prec b$, and $a \neq b$; we have to show that $a \prec a \jj b$.

Without loss of generality, we can assume that $a$ is to the left of $b$. Consider a cell $X$ with zero $o$ containing $a$ on its left chain; there is such a cell because $b$ is to the right of $a$. 

If $b$ is on the right chain of $X$, then $a \jj b  =  1_X$ and $a \prec 1_X$, since $X$ is a $4$-cell. If $b$ is not on the right chain of $X$, let $a_1$ be the element on the right chain of~$X$ covering $o$. Symmetrically let $Y$ be a $4$-cell with $b$ on the right chain of $Y$. Let $b_1\parallel b$ covering $o$ be on the left chain of $Y$. By our assumption, $1_X=1_Y$ and therefore $a\prec 1_X= a\jj b$.
\end{proof}

In light of Lemma~\ref{L:4cell2}, the following configuration is crucial for our investigations:

\begin{definition}\label{D:sideAdj}
Let $L$ be a $4$-cell lattice. The cells $A$ and $B$ in $L$ are \emph{upper-adjacent}, if $1_A  =  1_B$ and there exists $u\in A$, $B$ such that $1_B=1_A\succ u\succ 1_A$, $1_B$.
\end{definition}

In other words, the cells $A$ and $B$ are upper-adjacent, if $A \ii B  =  \set{1_A  =  1_B, u}$, where $u$ is common atom of $A$ and $B$, as illustrated in Figure~\ref{Fi:uapair}. If we denote by $U$ these  cells, then $1_U$ stands for $1_A  =  1_B$.

\begin{figure}[htb]
\centerline{\includegraphics{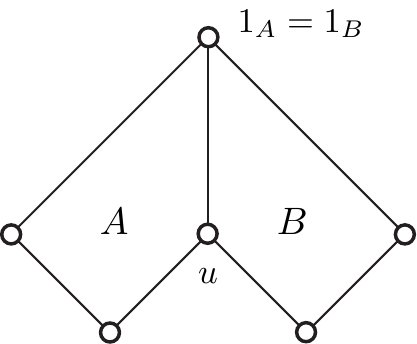}}
\caption{Upper-adjacent cells.}\label{Fi:uapair}
\end{figure}

\begin{lemma}\label{L:sideAdj}
Let $L$ be a planar semimodular lattice. Then $L$ is modular if{}f there are no upper-adjacent cells.
\end{lemma}
\begin{proof}
Let $L$ be a planar semimodular lattice. By Lemma~\ref{L:4cell1}, $L$ is a $4$-cell lattice. Let us assume that $L$ is nonmodular, that is, $L$ is not lower semimodular. By the dual of Lemma~\ref{L:4cell2}, there are cells $A$ and $B$, with $A$ to the left of $B$, satisfying $1_A  =  1_B$ and $0_A \neq 0_B$. Let $\set{0_A, a, 1_A}$ be the left chain of $A$, let $\set{0_B, b, 1_B}$ be the right chain of $B$, and let $\set{0_B, u, 1_B}$ be the left chain of $B$. If $A$ and $B$ are upper-adjacent, then we are done. Otherwise, if $A\ii B = \set{1_A = 1_B}$, then let $c\prec 1_B$ be the rightmost element to the left of $u$ such that $0_B\not\prec c$. Let $v\prec 1_B$ be to the leftmost element to the right of $c$. Define $d \prec 1_B$ to be the leftmost element to the right of $v$. Now we define two cells: $C$, the cell formed by $\set{1_B, v, c, v \mm c}$ and  $D$, the cell formed by $\set{1_B, v, d, v \mm d}$. Then $C$ and $D$ are upper-adjacent.
\end{proof}

\begin{lemma}\label{L:4cell3}
Let $L$ be a slim $4$-cell lattice. Then $L$ is semimodular if{}f $0_A  =  0_B$ implies that $A  =  B$, for cells $A$ and $B$ of $L$.
\end{lemma}

\begin{proof}
This follows immediately from Lemma~\ref{L:4cell2}, which states if $0_A = 0_B$, then $1_A = 1_B$, and from the fact that for two distinct cells $A$ and $B$, there is a doubly irreducible element in the interior of $A\uu B$.
\end{proof}

\begin{lemma}\label{L:4cell4}
Let $L$ be a slim semimodular lattice. Then exactly one of the following holds, for all $x \in L$:
\begin{enumeratei}
\item $x  =  1$;
\item $x$ is covered by a \emph{unique} element \tup{(}denoted by $x^*$\tup{)};
\item $x$ is covered by \emph{exactly two} elements \tup{(}denoted by $x^*_L$ and $x^*_R$, where $x^*_L$ is to the left of $x^*_R$\tup{)}.
\end{enumeratei}
\end{lemma}

\begin{proof}
Let us assume that $x$ is covered by $y$, $z$, $t$, where $y$ is immediately to the left of $z$ and $z$ is immediately to the left of $t$. Then $A  =  \set{x, y, z, y \jj z}$ and $B  =  \set{x, z, t, z \jj t}$ are distinct $4$-cells with $0_A  =  0_B  =  x$. $L$ is not slim by the contrapositive of Lemma~\ref{L:4cell3}.
\end{proof}

If $x$ is covered by a unique element, then $x^*_L$ and $x^*_R$ both are defined as $x^*$. 

\begin{lemma}\label{L:4cell5}
Let $L$ be a slim semimodular lattice. Let $x \in L$. If $a$, $b$, $c \prec x$ are distinct, then there exists an upper-adjacent pair $U$ such that $1_U=x$.
\end{lemma}

\begin{proof}
Let $a$, $b$, $c\prec x$ with $a$ immediately to the left of $b$ and $b$ immediately to the left of $c$. Define the cells $A=\set{x,a,b,a\mm b}$ and $B=\set{x,b,c,b \mm c}$. Then $U=A \uu B$ is an upper-adjacent pair with $1_U=x$.
\end{proof}

\begin{definition}
Let $L$ be a slim semimodular lattice and let $U$ be an upper-adjacent pair in $L$. Let us call $U$ \emph{maximal}, if $1_V > 1_U$ fails, for any upper-adjacent pair $V$ in $L$.  
\end{definition}

$U$ may not be unique as demonstrated in Figure~\ref{Fi:twomaximalua}.

\begin{figure}[htb]
\centerline{\includegraphics{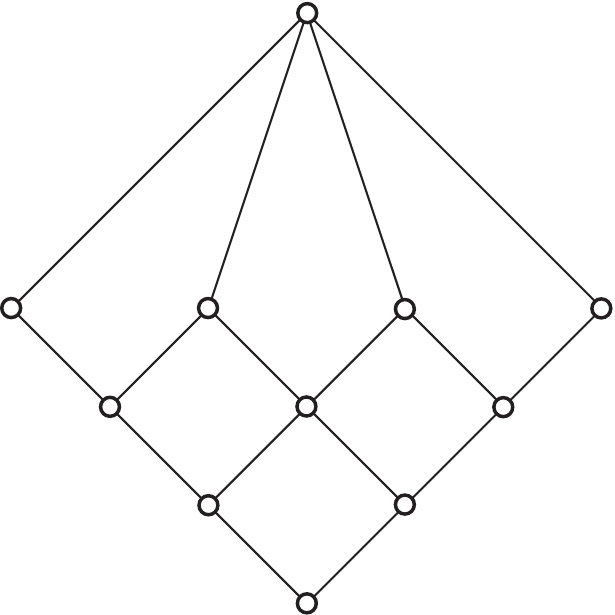}}
\caption{A lattice with two upper-adjacent maximal pairs.}\label{Fi:twomaximalua}
\end{figure}

\section{More on upper-adjacent cells}\label{S:upperadjacent}
In the next few sections we deal with a slim semimodular lattice $L$. A lattice is planar if it has a planar diagram. For $L$, we also fix a planar diagram representing it. For any sublattice $K$ of $L$, we have a planar diagram representing it, the one we obtain from the planar diagram of $L$. When we speak of $K$, we really have this planar representation in mind.

Let $U$ be a pair of maximal upper-adjacent $4$-cells in $L$, labeled as in Figure~\ref{Fi:uapair-labeled}. We~associate with $U$ two chains:

\begin{figure}[htb]
\centerline{\includegraphics{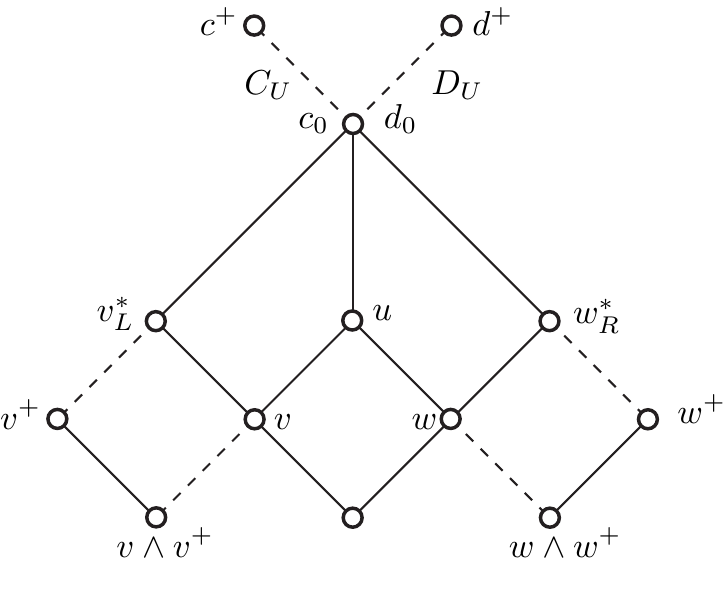}}
\caption{The labeled upper-adjacent pair in $L$.}\label{Fi:uapair-labeled}
\end{figure}

\begin{definition}\label{D:leftchain}
The \emph{left-chain $C_U$ associated with $U$} is a maximal chain 
\[
    c_0 \prec c_1 \prec \dots \prec c_k
\]
in $L$ such that $c_0=1_U$ and $C_U$ has the following two properties, for every $1 \leq i < k$:
\begin{enumeratei}
\item $c_i  =  (c_{i-1})^*_L$;
\item there is a $4$-cell $C^U_i  =  \set{x_{i-1}, x_i, c_{i-1}, c_i}$ with zero $x_{i-1}$, unit $c_i$, with $x_i$ to the left of $c_{i-1}$.
\end{enumeratei}
\end{definition}

Obviously, $c_0 \in C_U$. The chain terminates with $c_k$, if $c_k =1$ or if $c_k <1$ but $c_k$ covers no element to the left of $c_{k-1}$. Intuitively $c_k$ is on the left boundary and $c_i$ is not on the left boundary for all $i<k$. 

Similarly, we define the \emph{right-chain $D_U$ associated with $U$}: 
\[
    d_0 \prec d_1 \prec \dots \prec d_l,
\]
and the $4$-cell $D^U_i$.

Of course, in general, $k$ and $l$ are distinct, and later it will be shown that the sets $C_U-\set{c_0}$ and $D_U-\set{d_0}$ are distinct.  

Notice to what extent $C_U$ and $D_U$ depend not only on the structure of $L$ but on the planar representation we fixed.

\begin{lemma}\label{L:intchains}
Let $L$ be a slim semimodular lattice and let $U$ be a \emph{maximal} upper-adjacent pair in $L$. Then the chains $C_U$ and $D_U$ are intervals given by $C_U=[c_0,c_k],D_U=[d_0,d_l]$.
\end{lemma}

\begin{proof}
We proceed in hope of a contradiction. Let $x\prec c_i$ for some $i$ such that $x\geq c_0$ and $x\neq c_{i-1}$. Since $x\neq c_{i-1}$ we have that $x$ is to the right of $c_{i-1}$. By the construction of $C_U$ there exists some $y\prec c_i$ to the left of $c_{i-1}$. Since we have distinct elements $x,y,c_{i-1}\prec c_i$ by Lemma~\ref{L:4cell5} there exists an upper adjacent pair with unit $c_i$ which contradicts the statement $U$ is maximal.
\end{proof}

\begin{lemma}\label{L:sidetops}
Let $L$ be a slim semimodular lattice and let $U$ be a \emph{maximal} upper-adjacent pair in $L$. Then $I$ has the following property, for each $c_i$ where $0<i\leq k$ we have that if $x$ is the rightmost element satisfying $x\prec c_i$ then $x=c_{i-1}$. By symmetry the same holds for $d_i$ where $0< i\leq l$, that is $d_{i-1}$ is the leftmost element covered by $d_i$.
\end{lemma}
\begin{proof}
Assume the contrary, that there for some $i>0$ we have that $x\prec c_i$ such that $x$ is to the right of $c_{i-1}$. By the definition of $C_U$, there exists $y\prec c_i$ to the left of $c_{i-1}$. We have distinct elements $y$, $c_{i-1}$, $x\prec c_i$ so by Lemma~\ref{L:4cell5} there is an upper-adjacent pair $V$ such that $c_i=1_V>1_U$. Since $U$ is maximal we have a contradiction.
\end{proof}

Next we define some elements of $L$. Let $c^+=c_k$ and $d^+=d_l$; see Figure~\ref{Fi:uapair-labeled}. The elements we next define are crucial in the decomposition of $L$.

\begin{definition}\label{D:vplus}
Define the element $v^+$ as the minimal element on the left boundary of $L$ such that $v^+\not\leq v$.
\end{definition}

\begin{lemma}\label{L:v+unique}
$v^+$ is uniquely defined.
\end{lemma}

\begin{proof}
$v^+$ is defined since $v\not\leq v^*_L$. Let $x$ and $y$ be two minimal elements satisfying $x \not\leq v$ and $y \not\leq v$ on the left boundary. Since both are on the left boundary we have either $x\leq y$ or $y\leq x$. Since both are minimal we have $x=y$.
\end{proof}

By symmetry, we define $w^+$, see Figure~\ref{Fi:uapair-labeled}.

We also associate with $U$ four intervals, $I_U = [v^+, c^+]$, $J_U = [w^+, d^+]$, $T_U = [c_0 = d_0, 1]$, and $B_U = [0, u]$; see Figure~\ref{Fi:fourIntervals}.

The following statement is crucial to our proof.

\begin{lemma}\label{L:almostdisj}
Let $L$ be a slim semimodular lattice and let $U$ be a \emph{maximal} upper-adjacent pair in $L$. Then
\[
   I_U \ii J_U = \set{c_0 = d_0}.
\]
\end{lemma}

\begin{proof}
Let us assume, to the contrary, that $c_i =d_j \in I_U \ii J_U$, for some $1 \leq i \leq k$ and some $1\leq j\leq l$. Using the notation of Definition~\ref{D:leftchain}, form the upper-adjacent pair $V$ of the $4$-cells $C^U_i$ and $D^U_j$. Clearly, $1_U = c_0 = d_0 <  c_i = d_j = 1_V$, contradicting the maximality of $U$.
\end{proof}

We need one final concept:

\begin{definition}
The pair of elements $x$ and $y$ of $L$ is called an $I$-\emph{bridge}, if $x \in B_U$, $y \in I_U$ and $x \prec y$. Symmetrically, we define a $J$-\emph{bridge}. A \emph{bridge} is an $I$-bridge or a $J$-bridge.
\end{definition}

Figure~\ref{Fi:fourIntervals} illustrates this concept. The following statement is obvious:

\begin{lemma}\label{L:special}
Let $x$ and $y$ be an $I$-bridge. Then for every $z \in L$, either $x \jj z = y \jj z$ or $x \jj z$ and $y \jj z$ is an $I$-bridge. And symmetrically, for a $J$-bridge.
\end{lemma}

\section{Special join-homomorphisms}\label{S:special}
We would like to construct planar semimodular lattices from planar distributive lattices. Since homomorphisms preserve distributivity, it is logical to try join-homomorphisms. Unfortunately, $N_5$ is a join-homomorphic image of $C_2 \times C_3$, so we have to look for special join-homomorphisms. Such a concept is introduced in the next definition.

\begin{definition}
Let $L$ and $K$ be finite lattices. We call the join-homomorphism $\gf \colon L \to K$ \emph{cover-preserving} if{f} it preserves the relation $\preceq$; equivalently, if $x \prec y$ implies that $x\gf \prec y\gf$, for all $x$, $y\in L$, provided that $x\gf \neq y\gf$.
\end{definition}

The next two lemmas show that this is the property we need.

\begin{lemma}\label{L:cpjh}
Let $\gf \colon L \to K$ be a cover-preserving join-homomorphism. If $x \prec y$ in $K$, then there exists $a \in \set{x}\gf^{-1} \ci L$ and $b \in \set{y}\gf^{-1} \ci L$ such that $a \prec b$ in $L$.
\end{lemma}

\begin{proof}
$\set{x}\gf^{-1}$ is a join-closed subset of $L$; let $a$ be its largest element. For any $b \in \set{y}\gf^{-1}$, observe that $(a \jj b)\gf  =  a\gf \jj b\gf  =  x \jj y  = y$, so we can choose a minimal element $b$ of $\set{y}\gf^{-1}$ with $a \le b$. We claim that $a \prec b$. Indeed, if $c \in L$ satisfies that $a \le c\le b$, then $a \gf \prec b\gf$ implies that $c \in \set{x}\gf^{-1}$ or $c \in \set{y}\gf^{-1}$. The first possibility yields that $a  =  c$ by the maximality of $a$, the second possibility yields that $b  =  c$ by the minimality of $b$.
\end{proof}

\begin{lemma}\label{L:semimodquotient}
Let $\gf$ be a cover-preserving join-homomorphism of the lattice $L$ onto the lattice $K$. If $L$ is semimodular, then so is $K$.
\end{lemma}

\begin{proof}
Let $x$, $y$, $z \in K$; let us assume that $x \prec y$. By Lemma~\ref{L:cpjh}, there exist $a\in \set{x}\gf^{-1}$ and $b\in \set{y}\gf^{-1}$ such that $a \prec b$. Let us choose an arbitrary $c \in \set{z}\gf^{-1}$.

Since $a \prec b$, the semimodularity of $L$ implies that $a \jj c \preceq b \jj c$. The map $\gf$ is a cover-preserving join-homomorphism, so $a\gf \jj c\gf \preceq b\gf \jj c\gf$, that is, $x \jj z \preceq y \jj z$.
\end{proof}

\section{One-step expansion}\label{S:oneexpansion}
We start with a slim semimodular lattice $L$. If $L$ is modular, then it is distributive, and we are on familiar territory. If $L$~is not modular, by Lemma~\ref{L:sideAdj}, $L$~has at least one pair $U$ of upper-adjacent $4$-cells. The crucial step is the One-step Expansion Theorem, which eliminates one such maximal pair.

\begin{theorem}[One-step Expansion Theorem]\label{T:exptheorem}
Let $L$ be a slim semimodular lattice. If $L$ is not modular, then there exists a slim semimodular lattice $\ol{L}$ with the following two properties:
\begin{enumeratei}
\item there is a cover-preserving join-homomorphism of $\ol{L}$ onto $L$;
\item $\ol{L}$ has one fewer pair of upper-adjacent $4$-cells than $L$.
\end{enumeratei} 
\end{theorem}

In this section, we construct $\ol{L}$. The smallest example is $L = S_7$ of Figure~\ref{Fi:semimod}. Starting with $S_7$, we construct the distributive lattice $C_3^2$ and a cover-preserving homomorphism $\gf$ of $C_3^2$ onto $S_7$. The classes of the congruence kernel of $\gf$ is marked by wavy lines.

\begin{figure}[htb]
\centerline{\includegraphics{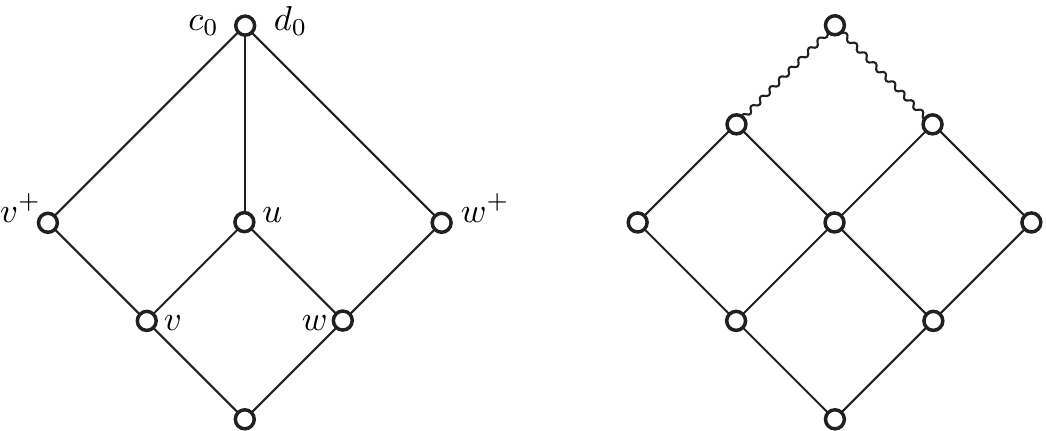}}
\caption{$S_7$ and $\ol S_7$.}\label{Fi:eg1}
\end{figure}

Since $L$ is not modular, by Lemma~\ref{L:sideAdj}, $L$ has upper-adjacent $4$-cells; let $U$ be a \emph{maximal} upper-adjacent $4$-cell of $L$; we keep $U$ fixed.

In Section~\ref{S:upperadjacent}, we associated with $U$ four intervals of $L$: $I_U$, $J_U$, $T_U$, and $B_U$ and two chains: $C_U$ and $D_U$; since $U$ is fixed, we shall denote them by $I$, $J$, $T$, $B$, $C$, and $D$.

\begin{lemma}\label{L:intervalsOnto}\hfill
\begin{enumeratei}
\item $L = T \uu B \uu I \uu J$.
\item $B$ is disjoint to $I$, $J$, and $T$.
\item $I \ii J = \set{c_0 = d_0}$.
\item $I \ii T = C$ and $J \ii T = D$. 
\end{enumeratei}
\end{lemma}
\begin{figure}[thb]
\centerline{\includegraphics{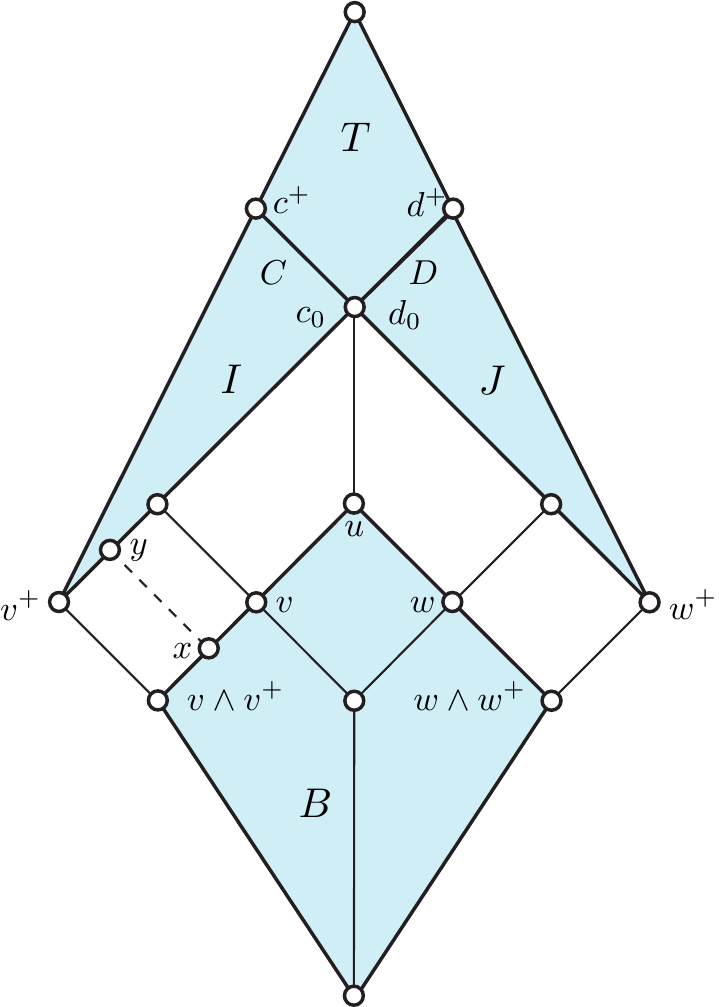}}
\caption{The four intervals of $L$.}\label{Fi:fourIntervals}
\end{figure}

\begin{proof}\hfill

(i). Let $x \in L$. We prove that $x\in T \uu B \uu I \uu J$.

If $x \leq u$, then $x \in B$. If $x > u$, then $x \geq c_0 = d_0$, so $x \in T$.

So we can assume that $x \parallel u$; by symmetry, we can also assume that $x$ is to the left of $u$. By Lemma~\ref{L:v+unique}, we obtain that $x \geq v^+$. If $x \leq c^+$, then $x \in I$ and we are done.

By way of contradiction, let us assume that $x \nleq c^+$, that is, $x$, $c^+ < x \jj c^+$. Let $z_0 = c^+ \prec \dots \prec z_m = x\jj c^+ \nleq c^+$ be a maximal chain in $L$ between $c^+$ and $x\jj c^+ \nleq c^+$; of course, $m \geq 1$. Let $y_{m-1} \prec z_m$ be the rightmost element to the left of $z_{m-1}$. Then $C_m = \set{y_{m-1} \mm z_{m-1}, y_{m-1},z_{m-1}, y_m}$ is a $4$-cell. Define $y_{m-2} = y_{m-1} \mm z_{m-1}$. By~induction, we get the $4$-cell $C_i = \set{y_{i-1} \mm z_{i-1}, y_{i-1},z_{i-1}, z_i}$ and the element $y_i$, for all $0 < i \leq m$. 

The $4$-cell $C_1$ by Definition~\ref{D:leftchain} (see also the comment following the definition) contradicts that $C$ terminates with $c^+$. This completes the proof of (i).

(ii) is obvious. 

(iii) was proved in Lemma~\ref{L:almostdisj}.

(iv) follows from Lemma~\ref{L:intchains}.
\end{proof}

Again note that this decomposition of $L$ and the definition of $\ol{L}$ depends on the planar representation of $L$ we fixed.

Now we are ready to define the lattice $\ol{L}$ of Theorem~\ref{T:exptheorem}.

\begin{figure}[htb]
\centerline{\includegraphics{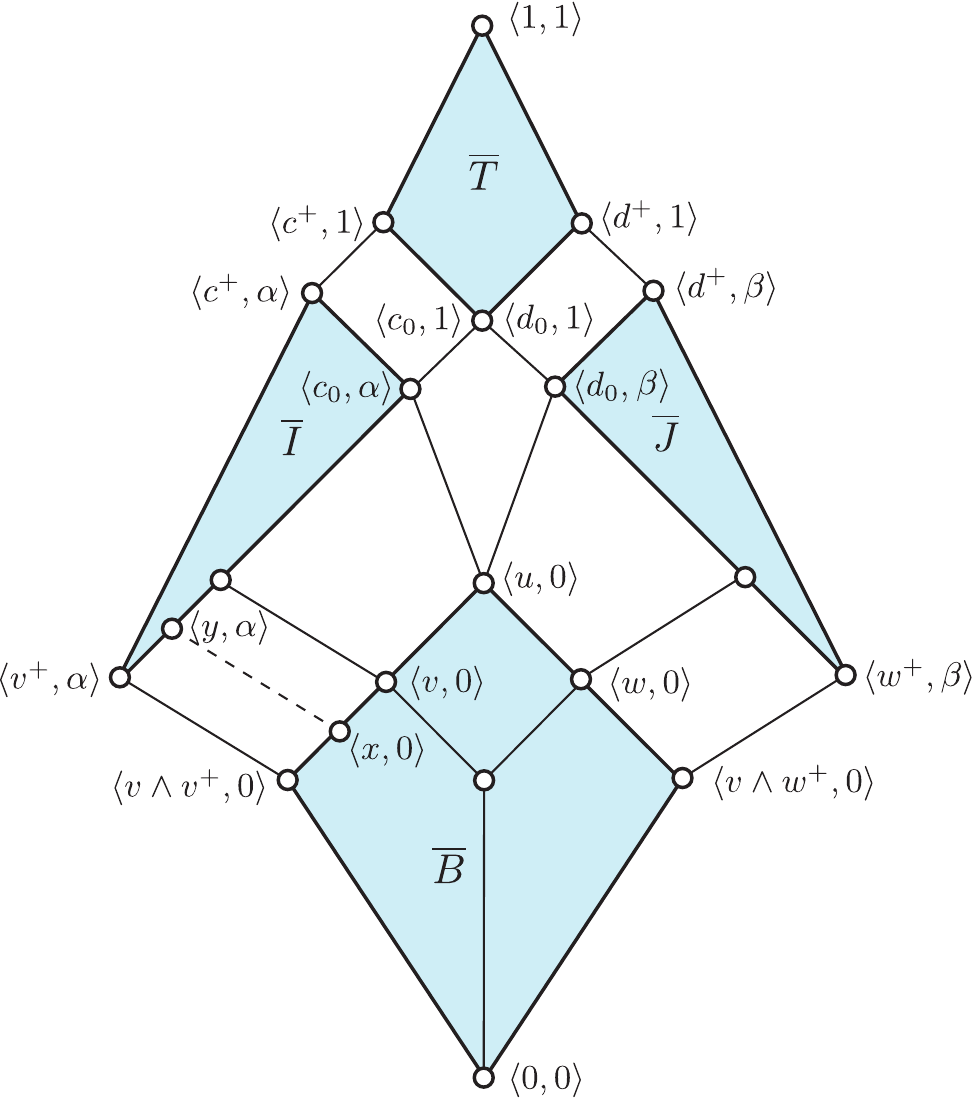}}
\caption{The one step extension $\ol L$ of $L$.}\label{Fi:extension}
\end{figure}

\begin{lemma}\label{L:explattice}
Let $C_2^2  =  \set{0, \ga, \gb, 1}$. We define four intervals of the lattice $L \times C_2^2$:
\begin{align*}
   \ol{T} & =  [\vv<c_0, 1>,\vv<1, 1>]=T \times\set{1},\\
   \ol{B} & =  [\vv<0,0>, \vv<u, 0>]=B \times\set{0},\\
   \ol{I} & =  [\vv<v^+,\ga>, \vv<c^+,\ga>]=I \times \set{\ga},\\
   \ol{J} & =  [\vv<w^+,\gb>, \vv<d^+, \gb>]=J \times\set{\gb}.
\end{align*}
Then $\ol{L}=\ol{T} \uu \ol{B} \uu \ol{I} \uu \ol{J}$ is a join-subsemilattice of $L \times C_2^2$ with zero $\vv<0, 0>$, hence $\ol{L}$ is a lattice.
\end{lemma}

\begin{proof}
Let $x  = \vv<a, \gg>$, $y = \vv<b,\gd> \in \ol{L}$. We show that $x \jj y \in \ol{L}$. We distinguish several cases.

Case 1. $\gg = \gd$. In this case, $a$ and $b$ are both in one $X$ of the four intervals $T$, $B$, $I$, or $J$ of $L$, so $a \jj b \in X\ci L$. It follows that $x\jj y\in \ol{L}$.

Case 2. $\gg = 0$ and $\gd = 1$; and symmetrically. Then $x \in B$ and $y \in T$ imply that $a \le u \le c_0 \le b$, so $a\jj b = b$ and $x \jj y = y \in \ol{T} \ci \ol{L}$.

Case 3. $\gg = 0$ and $\gd = \ga$; and symmetrically. Then $a \in B$, $b \in I$ so $a \jj b \in I$ and $x \jj y \in \ol{I} \ci \ol{L}$. Similarly for $\gd = \gb$.

Case 4. $\gg = 1$ and $\gd = \ga$; and symmetrically. Then $a \jj b \in T$ so $x \jj y \in \ol T \ci \ol{L}$. 

Case 5. $\gg = \ga$ and $\gd = \gb$. Then $x \in I$ and $y \in J$. So $a \jj b \in T$, and $x \jj y \in \ol T \ci \ol{L}$.
\end{proof}

\begin{lemma}\label{L:expcover}
Let $\vv<x, \gg>$, $\vv<y, \gd>\in \ol{L}$. Then $\vv<x, \gg> \prec\vv<y, \gd>$ in $\ol{L}$ if{}f one of the following two conditions holds:
\begin{enumeratei}
\item $\vv<x,\gg>\prec\vv<y,\gd>$ in $L\times C^2_2$;
\item $x$, $y$ is a bridge.
\end{enumeratei}
\end{lemma}
\begin{proof}
Let $\vv<x, \gg> \prec \vv<y, \gd>$ in $\ol{L}$. If $x \nin B$, then $y \nin B$, so $\vv<x, \gg>$, $\vv<y, \gd>\in \ol{I} \uu \ol{T}$ or $\vv<x,\gg>$, $\vv<y,\gd> \in \ol{J} \uu \ol{T}$; since $\ol{I} \uu \ol{T}$, $\ol{J} \uu \ol{T}$ are intervals of $L \times C^2_2$, we conclude that (i). If $x \in B$, then $\gg = 0$ and $\gd \in \set{0, \ga,\gb}$. If $\gd = 0$, then (i) holds. This leaves the case $\gd \in \set{\ga, \gb}$. By symmetry, we can assume that $\gd = \ga$. Let $z \in L$ satisfy $x \leq z \prec y$ in $L$; so $\vv<z, \gl> \in \ol{L}$, for some $\gl$. Since $z \prec y$ and $\vv<x, 0> \leq \vv<z, \gl> \prec \vv<y, \ga>$, these imply that $z = x$ and $\gl = 0$, and so $x \prec y$ in $L$ and so $x$, $y$ is a bridge.

Conversely, if (i) holds, then  $\vv<x, \gg> \prec\vv<y, \gd>$ in $\ol{L}$ is obvious. Let us assume that (ii) holds. By symmetry, we can assume that $\gd = \ga$. Then $x \in B$ and $y \in I$. In~$L \times C_2^2$, the interval $[\vv<x, \gg>, \vv<y, \gd>]$ is $\set{\vv<x, \gg>, \vv<x, \gd>, \vv<y, \gg>, \vv<y, \gd>}$, so the interval $[\vv<x, \gg>, \vv<y, \gd>]$ in $\ol{L}$ is $\set{\vv<x, \gg>, \vv<y, \gd>}$, that is, $\vv<x, \gg> \prec\vv<y, \gd>$ in $\ol{L}$.
\end{proof}

Now we state and prove the second crucial property of $\ol{L}$.

\begin{lemma}\label{L:expsm}
$\ol{L}$ is semimodular.
\end{lemma}

\begin{proof}
Let $\vv<x, \gg>$, $\vv<y, \gd>$, $\vv<z, \gl> \in \ol{L}$. Let $\vv<x, \gg> \prec \vv<y, \gd>$ in $\ol{L}$, in particular, $\gg \leq \gd$. Either condition (i) or condition (ii) of Lemma~\ref{L:expcover} holds. Without loss of generality, we also assume that $\vv<x, \gg> \leq \vv<z, \gl>$; in particular, $\gg \leq \gl$.

Let condition (i) of Lemma~\ref{L:expcover} hold. Then $\vv<x, \gg>\prec \vv<y, \gd> \in L\times C^2_2$. Since $L \times C^2_2$ is semimodular and $\ol{L}$ is a join-subsemilattice of $L\times C^2_2$, it follows that $\vv<x, \gg> \jj \vv<z, \gl> \preceq \vv<y, \gd> \jj \vv<z, \gl>$ in $\ol{L}$.

Let condition (ii) of Lemma~\ref{L:expcover} hold. Then semimodularity follows from Lemma~\ref{L:special}.
\end{proof}

\begin{lemma}\label{L:expplanar}
$\ol{L}$ is planar.
\end{lemma}
\begin{proof}
A sketch of a planar diagram of $\ol L$ is shown in Figure~\ref{Fi:extension}. We also have to connect all $\vv<c_i, \ga>$ with $\vv<c_i, 1>$, for $1 \leq i \leq k$ and all $\vv<d_i, \gb>$ with $\vv<d_i, 1>$, for $1 \leq i \leq~l$. By the construction of $C$ we know that $C$ is in the bottom left of $T$, by Lemma~\ref{L:sidetops} we have that $C$ is in the upper right of $I$ as represented in our sketch. Similarly for $D$.
\end{proof}

\begin{lemma}\label{L:expslim}
$\ol{L}$ is slim.
\end{lemma}
\begin{proof}
Obvious.
\end{proof}

\begin{lemma}\label{L:expjh}
The map $\vv<a,\gamma>\gf  = a$ is a cover-preserving join-homomorphism such that $\ol{L}\gf =L$.
\end{lemma}

\begin{proof}
The first projection of $L \times C^2_2$ onto $L$ is a join-homomorphism such that $(L \times C^2_2)\gf = L$ and $\ol L$ is a join-subsemilattice of $L \times C^2_2$. $\gf$ is cover-preserving by Lemma~\ref{L:expcover}.
\end{proof}

\section{The Expansion Theorem}\label{S:expansion}
The following lemma leads us to the Expansion Theorem.

\begin{lemma}\label{L:expfewerpairs}
$\ol{L}$ contains one fewer upper-adjacent pairs than $L$.
\end{lemma}
\begin{proof}
Let us define a function $\gy$ mapping upper-adjacent pairs of $\ol{L}$ to upper-adjacent pairs of $L$: for an upper-adjacent pair in $\ol L$, let $V\gy = V\gf$.

To show $\gy$ is well defined let $V$ be an upper adjacent pair in $\ol{L}$ to show that $V\gy$ is an upper-adjacent pair in $L$ distinguish the following cases:

Case 1. $1_V\leq\vv<c^+,\ga>$; and symmetrically. Then $V\gy\ci I\uu B$ is an upper-adjacent pair in $L$ since when restricted to $I\uu B$ the map $\gy$ is exactly the first projection and an isomorphism.

Case 2. $1_V=\vv<c_i,1>$ for some $i\leq k$; and symmetrically. Then $1_V$ covers exactly two elements, a contradiction.

Case 3. $1_V>\vv<c_i,1>$ for all $i\leq k$ and $1_V>\vv<d_i,1>$ for all $i\leq l$. Then for all $y\in V$ we have $y\geq c_0$. So $V\ci \ol{T}$ and $V\gy\ci T$ is an upper-adjacent pair in $L$ since when restricted to $T$ the map $\gy$ is exactly the first projection and an isomorphism.

To show $\gy$ is one-to-one, let $V,W\ci \ol{L}$ be two upper-adjacent pairs with interior atoms $\vv<a,\gg>$, $\vv<b,\gd>$ respectively such that $V\gf=W\gf$. So $a=\vv<a,\gg>\gf=\vv<b,\gd>\gf=b$. If $V$, $W$ are distinct then $a=b\in T$ which contradicts the statement of $U$ is maximal in the construction of $\ol{L}$.

$\set{u}\gf^{-1}=\vv<u,0>$ has two distinct covers and cannot be the interior atom of any upper-adjacent pair. Thus $\ol{L}$ contains at least one fewer upper-adjacent pairs than~$L$.
\end{proof}

Now we are ready for the Expansion Theorem.

\begin{theorem}\label{T:002}
Let $L$ be a slim semimodular lattice. There exists a planar distributive lattice $D$ and a cover-preserving join-homomorphism $\gf$ such that $L = D\gf$.
\end{theorem}

\begin{proof}
Apply the One-step Extension Theorem, to obtain a sequence of lattices $L = L_0,\dots,L_k = D$ and a sequence of cover-preserving join-homomorphisms $\gf_1,\dots,\gf_k$ such that for all $i$ we have $L_{i-1} = L_i\gf_i$ and $L_k$ has no pairs of upper-adjacent $4$-cells. $k\le$ the number of upper-adjacent pairs of $4$-cells in $L$. Since $L_k = D$ has no upper-adjacent $4$-cells it is modular and since it's also slim it is distributive. Let $\gf = \gf_1\dots \gf_k$ which is cover-preserving.
\end{proof}

\section{An example}\label{S:eg}
Figure~\ref{Fi:semimod} shows the smallest example of our construction. 

To provide a less trivial example, let $L$ be the lattice with the planar diagram on the left of Figure~\ref{Fi:eg2-1}. This lattice contains two upper-adjacent pairs.

\begin{figure}[htb]
\centerline{\scalebox{.8}{\includegraphics{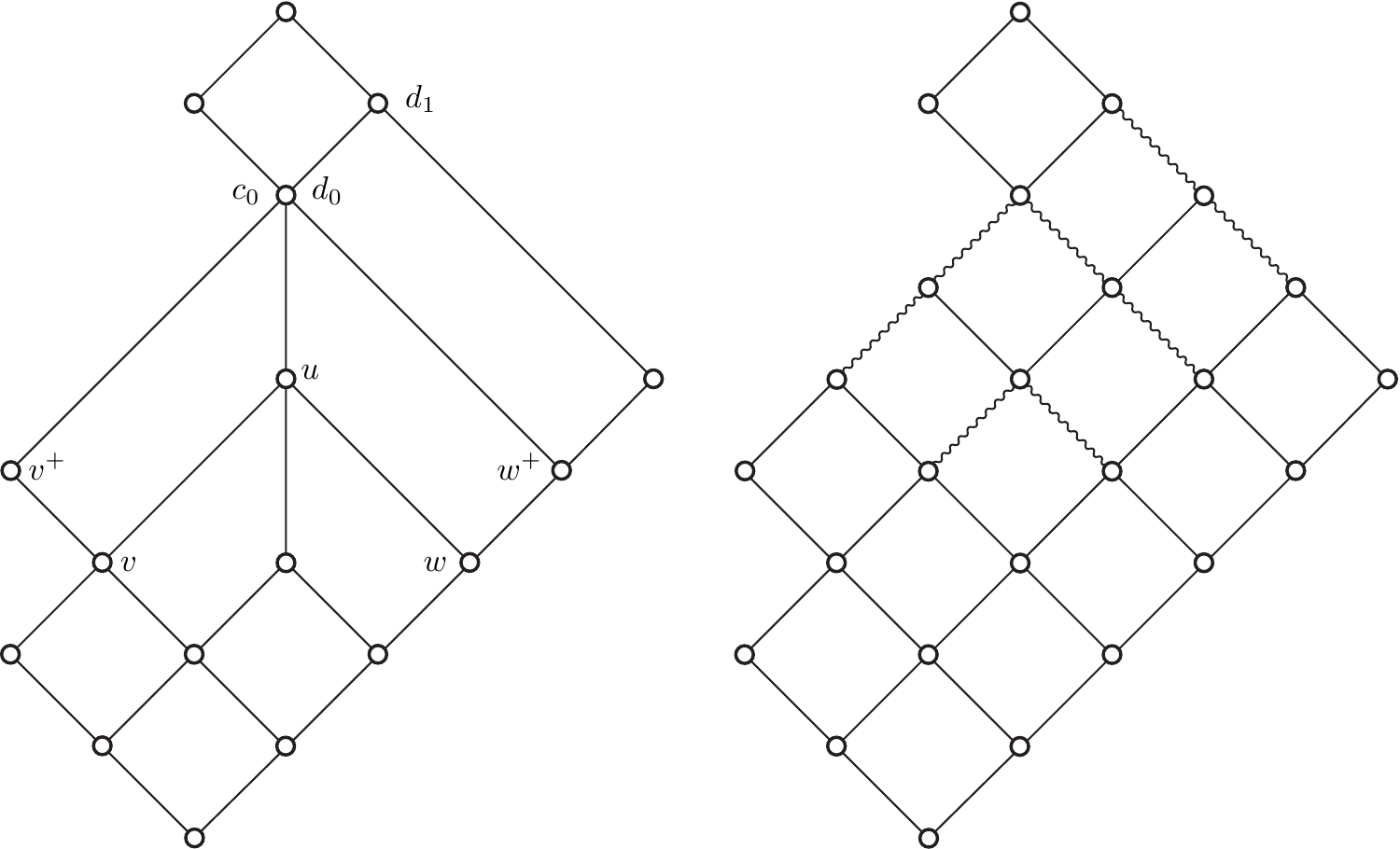}}}
\caption{A lattice with two upper-adjacent pairs and its extension.}\label{Fi:eg2-1}
\end{figure}

The expansion is the lattice on the right of Figure~\ref{Fi:eg2-1}; the classes of the congruence kernel of the cover-preserving join-homomorphism are marked by wavy lines.

\section{Conclusion}\label{S:conclusion}

Using the above results we have the following theorem:

\begin{theorem}\label{T:planarsm}
A planar semimodular lattice can be obtained from the direct product of two finite chains in the following three steps:

\begin{enumerate}[\upshape (1)]
\item Remove a left and a right corner \tup{(}possibly empty\tup{)} of the direct product of the chains to obtain a planar distributive lattice $D$.
\item Apply a cover-preserving join-homomorphism to $D$.
\item Add doubly-irreducible elements to the interiors of $4$-cells.
\end{enumerate}
\end{theorem}


\begin{thebibliography}{9}

\bibitem{GLT2}
G. Gr\"atzer,
General Lattice Theory, second edition. New appendices
by the author with B.\,A. Davey, R. Freese, B. Ganter, M. Greferath,
P. Jipsen, H.\,A. Priestley, H. Rose, E.\,T. Schmidt, S.\,E.
Schmidt, F. Wehrung, and R. Wille.
Birkh\"auser Verlag, Basel, 1998. xx+663 pp. ISBN: 0-12-295750-4, ISBN:
3-7643-5239-6.
\emph{Softcover edition,} Birkh\"auser Verlag, Basel--Boston--Berlin, 
2003.
ISBN: 3-7643-6996-5.

\bibitem{CLFL}
G. Gr\"atzer,
The Congruences of a Finite Lattice, A \emph{Proof-by-Picture} Approach.
Birkh\"auser Boston, 2005. xxiii+281 pp.\\ 
ISBN: 0-8176-3224-7.

\bibitem{KR75}
D. Kelly and I. Rival, 
\emph{Planar lattices,}
Canad. J. Math. \tbf{27} (1975), 636--665. 


\bibitem{sM43}
S. MacLane,  
\emph{A conjecture of Ore on chains in partially ordered sets,} 
Bull. Amer. Math. Soc. \tbf{49} (1943), 567--568.
09.1X

\bibitem{oO43}
O. Ore, 
\emph{Chains in partially ordered sets,} 
Bull. Amer. Math. Soc. \tbf{49} (1943), 558--566.

\end{thebibliography}
\end{document}